\documentclass[12pt]{amsart}

\usepackage{enumerate}
\usepackage{graphicx}

\usepackage[active]{srcltx}
\usepackage{nicefrac}

\usepackage{amsthm,amssymb,setspace}
\usepackage[pagebackref,colorlinks,linkcolor=red,citecolor=blue,urlcolor=blue,hypertexnames=true]{hyperref}
\usepackage{amsrefs}
\usepackage[matrix, arrow]{xy}

\newcommand{\sS}{\mathcal{S}}
\newcommand{\R}{\mathbb R}
\renewcommand\epsilon{\varepsilon}

\DeclareMathOperator*{\diag}{diag}
\DeclareMathOperator*{\interior}{int}
\DeclareMathOperator*{\cone}{cone}
\DeclareMathOperator*{\lin}{lin}

\theoremstyle{plain}
\newtheorem*{theorem*}{Theorem}
\newtheorem{theorem}{Theorem}[section]

\newtheorem{lemma}[theorem]{Lemma}

\theoremstyle{definition}
\newtheorem*{definition*}{Definition}

\theoremstyle{remark}

\newtheorem{remark}[theorem]{Remark}
\newtheorem{example}[theorem]{Example}

\begin{document}


\title{Smooth Hyperbolicity Cones are Spectrahedral Shadows}

\author{Tim Netzer}
\address{Tim Netzer, Universit\"at Leipzig, Germany}
\email{netzer@math.uni-leipzig.de}

\author{Raman Sanyal}
\address{Raman Sanyal, Fachbereich Mathematik und Informatik, %
Freie Universit\"at Berlin, %
Germany}
\email{sanyal@math.fu-berlin.de}


\date{\today}
\thanks{Raman Sanyal has been supported by the European Research Council under
the European Union's Seventh Framework Programme (FP7/2007-2013) / ERC grant
agreement n$^\mathrm{o}$ 247029.}

\begin{abstract} 
    Hyperbolicity cones are convex algebraic cones arising from hyperbolic
    polynomials. A well-understood subclass of hyperbolicity cones is that of
    spectrahedral cones and it is conjectured that every hyperbolicity cone
    is spectrahedral. In this paper we prove a weaker version of this
    conjecture by showing that every smooth hyperbolicity cone is the linear
    projection of a spectrahedral cone, that is, a spectrahedral shadow.
\end{abstract}

\maketitle

\section{Introduction} \label{sec:intro}

A homogeneous polynomial $h \in \R[x]=\R[x_1,\ldots,x_n]$ is called 
\emph{hyperbolic in direction $e\in\R^n$}, if $h(e) \neq 0$ and the univariate
polynomials 
\[
    h_{a,e}(t) \ := \ h(a-te) \ \in \ \R[t]
\]
have only real roots, for all $a\in\R^n$. So, geometrically, all the lines
parallel to $\R e$ meet the hypersurface associated to $h$ in only real
points.  Hyperbolic polynomials were first considered in the area of partial
differential equations, e.g.\ in~\cite{gar,lax}. Recently, there has been
renewed interest in hyperbolic polynomials in the areas of
optimization~\cite{gue,BGLS,ren}, convex algebraic
geometry~\cite{hevi,vin,br1}, and combinatorics~\cite{cosw,br2,gur}. The
connection to convex geometry was discovered by G\aa rding who proved that the
\emph{hyperbolicity cone} 
$$
\Lambda_e(h) \ := \ 
    \left\{ a\in\R^n \;:\; h_{a,e}(t) \text{ has only non-negative real roots} \right\}
$$
is a closed convex cone. An alternative characterization of $\Lambda_e(h)$ is
the closure of the connected component of $\{ h \not= 0 \}$ containing $e$.
G\aa rding also showed that $h$ is hyperbolic in all directions $e'\in
\interior \Lambda_e(h)$ and that the hyperbolicity cone is independent of
these directions, see also Renegar \cite{ren}. The importance of hyperbolic
polynomials to convex optimization was recognized by G\"uler~\cite{gue} how
showed that interior point methods used in convex programming work for
hyperbolicity cones by utilizing $-\log h(x)$ as a barrier function.

An important class of hyperbolic polynomials arises from \emph{definite
determinantal representations}.  If $M_1,M_2, \dots,M_n$ are hermitian matrices
with $e_1M_1+ e_2M_2 + \cdots +e_nM_n$ strictly definite, then
\[
    h \ = \ \det\left(x_1M_1 \ + \ x_2M_2 \ + \ \cdots \ + \ x_nM_n\right)
\]
is hyperbolic in direction $e$.  In fact, without loss of generality we 
 can assume that $e_1M_1+\cdots +e_nM_n=I$. It can be easily seen that 
$h_{a,e}(t)$ is then the characteristic polynomial of the
hermitian matrix $M(a) = a_1M_1+\cdots +a_nM_n$ and the hyperbolicity cone is 
\begin{equation}\label{spec}
    \Lambda_e(h) \ = \ \{a\in\R^n\;:\; M(a) \ = \ a_1M_1 \ + \ a_2M_2 \ + \
\cdots \ + \ a_nM_n \text{
    is positive semidefinite} \}.
\end{equation} 
A cone of this form is a linear section of the cone of (hermitian) positive
semidefinite matrices and is called a \emph{spectrahedral cone}. Spectrahedral
cones or, more precisely, spectrahedra are exactly the sets of feasible
solutions to semidefinite programs.  It is conjectured that the classes of
spectrahedral and hyperbolicity cones coincide.

\textbf{Generalized Lax Conjecture.} \emph{Every hyperbolicity
cone is spectrahedral, i.e., a linear section of the cone of positive
semidefinite matrices.}

In its original form, the Lax conjecture was stated for $n=3$, and proved by
Helton \& Vinnikov \cite{hevi} in an even stronger form; see also
\cite{lepara}. For $n \ge 4$, the conjecture is still wide open;
see~\cite{vin} for an up-to-date overview. 

Unlike polyhedra, the class of spectrahedra is not closed under projection.
The image of a spectrahedron under a linear projection is called a
\emph{spectrahedral shadow} or \emph{sdp-representable set}. Spectrahedral
shadows lack many of the desirable properties of spectrahedra. However, from a
practical viewpoint, spectrahedral shadows are very valuable as optimization
over a spectrahedral shadow can be done with semidefinite programming. In this
paper we prove a weaker version of the Generalized Lax Conjecture. 

\begin{theorem}\label{mainhyp}
    Let $h\in \R[x]$ be hyperbolic with respect to $e$. If each non-zero point
    in the boundary of $\Lambda_e(h)$ is a smooth point of $h$, then
    $\Lambda_e(h)$ is a spectrahedral shadow. 
\end{theorem}

We actually prove a stronger result for a larger class of hyperbolic
polynomials, but we defer the more technical assumptions (cf.\
Theorem~\ref{mainrz}). Let us emphasize that the assumption on $h$ is very
mild as it comprises all strictly hyperbolic polynomials ($h_{a,e}(t)$ has
only simple roots) which form a dense open subset among all hyperbolic
polynomials (in fixed dimension). This result of Nuij \cite{nuij} is explained
in more detail at the end of the paper. At this point we note that our results
imply that {\it any} hyperbolicity cone can be easily approximated arbitrarily
close by spectrahedral shadows.

For the proof, it will be sufficient to consider pointed hyperbolicity cones
and we pass to a dehomogenization $\sS$ of $\Lambda_{e}(h)$ which is a
compact, convex, and basic semi-algebraic set described by (inhomogeneous)
real-zero polynomials. We show that the describing polynomials are
\emph{strictly quasi-concave} in a neighborhood of every smooth point in the
boundary of $\sS,$ which enables us to use results and ideas of 
Helton \& Nie \cite{heni,heni2} to show that $\sS$ is a spectrahedral shadow.
These tools break down in the presence of more severe singularities (e.g.\
self-intersections) in the boundary of $\Lambda_{e}(h),$ but we conjecture that 
\emph{all} hyperbolicity cones are spectrahedral shadows. \\

\textbf{Acknowledgements.} We would like to thank Daniel Plaumann for many inspiring discussions on the topic.

\section{Real-zero polynomials and quasi-concavity} \label{sec:rz_quasi}
\newcommand\oh{\overline{h}}
\renewcommand\oe{\overline{e}}

A (possibly inhomogeneous) polynomial $p \in \R[x]$ is called
\emph{real-zero with respect to $e\in\R^{n}$}, if $p(e)\neq
0$ and the univariate polynomials $p_{a,e}(t):=p(e+ta)\in \R[t]$ have
only real roots, for all $a\in\R^n$. So the defining property of real-zero
polynomials is similar to that of hyperbolic polynomials, the main difference
being that now every line \emph{through} $e$ has to meet the
hypersurface of $p$ in only real points.  The precise relation to hyperbolic
polynomials is as follows: If $p$ is a real-zero polynomial with respect to $e$ of
degree $d$, then the homogenization $h(x_0,x) := x_0^d\, p(\frac{x}{x_0} )$ is
hyperbolic wrt to $(1,e)$.  Conversely, if $h$ is hyperbolic with
respect to $e \in \R^n$, then the restriction of $h$ to any hyperplane $H$
containing $e$ is a real-zero polynomial with respect to $e \in H$.

For a real-zero polynomial $p$, the set 
\[
    \sS_e(p) \ = \ \{ a \;:\; p_{a-e,e}(t)\ \text{has no root in } [0,1)\}
\]
is called the {\it rigidly convex set of $p$ and $e$}. It is the closure of
the connected component of $\{p\neq 0\}$ containing $e$ and coincides with $H
\cap \Lambda_{e}(h)$. It follows that $\sS_e(p)$ is a closed convex set and
$p$ is real-zero with respect to every point in the interior of $\sS_e(p)$.

If there are hermitian matrices $M_0,M_1,\dots,M_n$ such that
$$
    p \ = \ \det\left( M_0 \ + \ x_1M_1\ + \ \cdots \ + \ x_nM_n \right)
$$
and $M_0+e_1M_1+\cdots +e_nM_n$ is positive definite, then $p$ is said to have a \emph{definite determinantal representation} and
the rigidly convex region
\[
    \sS_e(p) \ = \ \{a\in \R^n \;:\; M_0 \ + \ a_1M_1 \ + \ \cdots \ + \
    a_nM_n \text{ positive semidefinite} \}
\]
is called a \emph{spectrahedron}. Again, linear projections of spectrahedra are called {\it spectrahedral shadows}.

Let us recall the notion of multiplicity of a point $a$ with respect to a
polynomial $h$. We can write $h(x+a)$ as a sum of homogeneous terms
\[
    h(x+a) \ = \ h_0(x) \ + \ h_{1}(x) \ + \ \cdots \ + \ h_d(x)
\]
where $h_i(x)$ is homogeneous of degree $i,$ and the \emph{multiplicity} or
\emph{order of vanishing} of $h$ at $a$ is the smallest $m$ for which $h_m(x)
\not= 0$. Thus, a point $a$ lies on the hypersurface associated to $h$ if and
only if the multiplicity is positive. A point $a$ is a \emph{smooth point} of
$h$ if the multiplicity is $1$. Equivalently, let $v \in \R^n$ be a generic
direction and consider the univariate polynomial
\[
    h_{a,v}(t) \ = \ h(a - tv) \ = \ h_0(-v) \ + \ h_1(-v) t \ + \ \cdots \ + \
    h_d(-v) t^d.
\]
Then the multiplicity of $a$ equals the multiplicity of $h_{a,v}(t)$ at $t =
0$. Here, the genericity assumption means that $h_i(v) \not = 0$ whenever $h_i
\not= 0$. Clearly, the order of vanishing of $h_{a,v}(t)$ at $t=0$ can only be
larger for non-generic $v$.  The next lemma asserts that for a hyperbolic
polynomial, every direction in the interior of the hyperbolicity cone is
sufficiently generic. For points $a\in\Lambda_e(h)$,
this was shown by Renegar~\cite[Prop.~22]{ren}.

\begin{lemma}\label{mult}
    Let $h \in \R[x]$ be a hyperbolic polynomial with respect to $e$. For 
    $a \in \R^n$, the multiplicity of $a$ with respect to $h$
    equals the order of vanishing of $h_{a,f}(t)$ at $t = 0$, independent of
    the choice of $f \in
    \interior \Lambda_{e}(h)$.
\end{lemma}
\begin{proof}
    Let $a \in \R^n$, $v \in \interior \Lambda_{e}(h)$ generic, and $f \in
    \interior
    \Lambda_e(h)$ arbitrary. Consider the hyperbolic polynomial 
    \[
        g(r,s,t) \ = \ h(r a \ - \ s  v \ - \ t f ).
    \]
    The multiplicity of $a$ with respect to $h$ is thus the order of vanishing of $g(1,s,0)$ at $s
    = 0,$ and the claim is that this is the same as the order of vanishing of $g(1,0,t)$ at $t=0$. By the
    Helton-Vinnikov Theorem~\cite{hevi}, we have
    \[
        g(r,s,t) \ = \ \det( r A \ - \ s V \ - \ t F )
    \]
    where $V$ and $F$ are positive definite matrices. Hence, they
    have a Cholesky factorization $V = \bar{V}\bar{V}^t$ and  $F =
    \bar{F}\bar{F}^t$ and the order of vanishing is the dimension of the kernel
    of $\bar{V}^{-1}A\bar{V}^{-t}$ and $\bar{F}^{-1}A\bar{F}^{-t}$,
    respectively.
\end{proof}
\begin{remark}(i)
Note that a similar result for real-zero polynomials is immediately deduced: The usual multiplicity of $a$ with respect to $p$ coincides with the order of vanishing of $p_{a-f,f}(t)$ at $t=1$, for any $f$ in the interior of $\sS_e(p).$

(ii) Note that in view of Lemma \ref{mult}, the assumption from Theorem \ref{mainhyp} just means that for $0\neq a\in\partial\Lambda_e(h),$ the polynomial $h_{a,e}(t)$ has a simple zero at $t=0.$
\end{remark}

For a twice differentiable function $g : \R^n \rightarrow \R$ let us denote by
$\nabla g(a)$ the gradient of $g$ at $a$ and by ${\rm H}(g;a)$ the Hessian
matrix.  The function $g$ is called \emph{strictly quasi-concave} at a point
$a \in \R^n$ if the quadratic form $v \mapsto v^t {\rm H}(g;a) v$ is negative
definite on the orthogonal complement of $\nabla g(a)$. In formulas, this is
for every $v \in \R^n \setminus \{0\}$ we require
\[
    v^t \nabla g(a) \ = \ 0 \quad \Rightarrow\quad  v^t {\rm H}(g;a) v \ < \ 0.
\]
The notion of strict quasi-concavity was introduced by Helton and
Nie~\cite{heni}, to give a condition when a basic semialgebraic set is a
spectrahedral shadow:
\begin{theorem}[{Helton \& Nie~\cite[Thm.~2]{heni}}]\label{henith} 
    Let $g_1,g_2,\ldots,g_m\in\R[x]$ and assume 
    $$
        S=\{ a\in\R^n\;:\; g_1(a)\geq 0,g_2(a) \ge 0,\ldots, g_m(a)\geq 0\}
    $$
    is compact and convex with nonempty interior. If each $g_i$ is strictly
    quasi-concave at each point of $S$, then $S$ is a spectrahedral shadow.
\end{theorem}

For polynomial functions $g$, strict quasi-concavity can be described as
follows. For $a \in \R^n$ write $g(x+a)$ in homogeneous terms as
\[
    g(x+a) \ = \ g_0(x) \ + \ g_1(x) \ + \ g_2(x) \ + \ \cdots  \ + \ g_k(x).
\]
Then $g$ is strictly quasi-concave at $a$ if for every $v \in \R^n \setminus
\{0\}$ with $g_1(v) = 0$, we have $g_2(v) < 0$. 
The following Lemma is key for showing that smooth hyperbolicity cones are
spectrahedral shadows.  Part (ii) is  already proven by elementary means in
\cite[Prop.~9]{ren}, we give a short alternative proof which is not elementary
however. For (i) we provide an elementary proof. Recall that a convex set $S$
is called pointed, if $S$ does not contain a line.

\begin{lemma}\label{sqc} 
    Let $p\in\R[x]$ be  real-zero with respect to $e$.
    \begin{itemize}
        \item[(i)] If $\sS_e(p)$ is pointed, then $p$
            is strictly quasi-concave at each interior point $a \in \sS_e(p)$.
        \item[(ii)] If $a\in\partial\sS_e(p)$ is a smooth point of $p$,
            and $p$ does not vanish on a full line through $a$, then $p$ is
            strictly quasi-concave at $a$.  
\end{itemize}
\end{lemma}

\begin{proof}
    For (i) let us  assume  without loss of generality, that $p(a) = 1$. Then
    \[
        q(t) \ := \ p(a + t v) \ = \ 1 \ + \ p_1(v) t \ + \ p_2(v)   t^2 \
        + \ \cdots \ + \ p_k(v) t^k \ = \ 
        \prod_{i=1}^k (1 + \lambda_i t)
    \]
    has only real non-zero roots. Note that $k\geq 1$ follows from the
    assumption that $\sS_e(p)$ does not contain a line. In particular,
    $p_1(v)  =  \lambda_1 + \lambda_2 + \cdots + \lambda_k$ and
    \[
        p_2(v) \ = \ \sum_{i < j} \lambda_i\lambda_j \ = \ 
        p_1(v)^2 \ - \ \sum_i \lambda_i^2.
    \]
    Hence $p_2(v) < 0$ whenever $p_1(v) = 0$ and $v \not= 0$.

    (ii) Assume $e=0$ and write 
    $$
        q(s,t) \ = \ p(sa+tv) \ = \ \det(I \ + \ sA \ + \ tV)
    $$ 
    with $A$ and $V$ being symmetric matrices of size  $d=\deg(q)$, using the
    Helton-Vinnikov theorem~\cite{hevi}.  From the fact that $a\in\partial
    \sS_e(p)$ is a smooth point of $p$ we deduce that $I+A$ is positive
    semidefinite of rank $d-1.$ Without loss of generality we may assume
    \mbox{$I+A=\diag(1,\ldots,1,0)$} and have $p(a+tv)=\det(I+A+tV)$. Let us
    write $V=(\ell_{ij})_{i,j}$. The Leibnitz formula now yields
    $$
        p_1(v) \ = \ \ell_{dd} \quad \text{and } \quad p_2(v) \ = \
        \sum_{i = 1}^{d-1}\ell_{ii}\ell_{dd}-\ell_{id}^2.
    $$
    So $p_1(v)=0$ and $p_2(v)\geq 0$ implies $\ell_{id}=0$ for all $i$. But
    then $p(a+t v) \equiv 0$ and $p$ vanishes on a line through $a$.    
\end{proof}

The following two examples show that the assumptions in Lemma~\ref{sqc} are
indeed necessary.
\begin{example}
    For the first part, consider the polynomial
    \[
    p(x,y,z) \ = \ \det \begin{pmatrix} z & x \\ x & 1 \end{pmatrix} \ = \ z - x^2 \
        \in \ \R[x,y,z]
    \]
    which is real-zero with respect to $e=(0,0,1).$ Its rigidly convex set is
    known as the \emph{Taco}. The strict quasi-concavity is not fulfilled at
    $e$.

    For the second part, consider the Cayley cubic
    \[
        p(x,y,z) \ = \ \det 
        \begin{pmatrix}
            1 & x & y \\
            x & 1 & z \\
            y & z & 1 
        \end{pmatrix}
        \ = \ 1+2xyz-x^2-y^2-z^2\in\R[x,y,z]
    \]
    which is an irreducible polynomial real-zero with respect to $e=(0,0,0)$.
    The corresponding rigidly convex set is known as the \emph{Samosa}. The
    boundary of the Samosa contains exactly $4$ singular points (of
    multiplicity $2$) and $\tbinom{4}{2} = 6$ line segments connecting pairs
    of singular points; see~\cite{sanyal} for an explanation of these numbers.
    In particular, every point in the interior of a line segment is smooth and
    $p$ is not strictly quasi-concave at these points.
\end{example}

\section{The Main Result} \label{mainsec} 

The first version of our main result is stated for real-zero polynomials.
\begin{theorem}\label{mainrz}
Let $p\in\R[x]$ be real-zero with respect to $e$, and assume $\sS_e(p)$ is compact. Further assume that for each $a\in\partial \sS_e(p)$ and each irreducible factor $p_i$ of $p$ with $p_i(a)=0$, $a$ is a smooth point of $p_i$ and $p_i$ does not vanish on a whole line through $a$. Then $\sS_e(p)$ is a spectrahedral shadow.
\end{theorem}

\begin{remark}\label{rem}
Note that the conditions in Theorem \ref{mainrz} are fulfilled if each point $a\in\partial\sS_e(p)$ is a smooth point of $p$. Then only one irreducible factor of $p$ can vanish on points of $\partial \sS_e(p)$, and compactness together with smoothness implies that this factor will not vanish on a whole line through any boundary point.
\end{remark}

\begin{proof}[Proof of Theorem \ref{mainrz}]
   Let $p=p_1\cdots p_m$ be the decomposition of $p$ into irreducible
   factors. Each $p_i$ is real-zero with respect to $e$, and
   $$
       \sS_e(p)=\sS_e(p_1) \ \cap \ \sS_e(p_2) \ \cap \   \cdots \ \cap \
       \sS_e(p_m).
   $$
   Fix some point $a\in\partial\sS_e(p)$ and let $I_a=\{ i\;:\; p_i(a)=0\}$.
   Since $a$ is a smooth point for every $p_i$ for $i \in I_a$, locally at
   $a$ the set $\sS_e(p)$ is defined by the conditions $p_i \ge 0$ for $i
   \in I_a$. That is, for $\epsilon >0$ samll enough
   \[
       N_\epsilon(a) \ := \ B_\epsilon(a) \ \cap \ \sS_e(p) \ = \
       \{b\in\R^n\;:\; \|
       b-a\|^2\leq\epsilon^2, p_i(b)\geq 0
       \text{ for } i\in I_a\}.
   \]
   We have $N_\epsilon(a) \subseteq \sS_e(p_i)$ for all $i \in I_a$ and, for
   $\epsilon > 0$ small, every point in $N_\epsilon(a) \cap
   \partial\sS_e(p_i)$ is a smooth point of $p_i$ for $i \in I_a$.  Moreover,
   for $\epsilon > 0$ sufficiently small, $N_\epsilon(a)$ misses all points of
   $\partial\sS_e(p_i)$ through which $p_i$ vanishes along a line. This also
   implies that $\sS_e(p_i)$ is pointed and we can apply Lemma~\ref{sqc} to
   infer that $p_i$ is strictly quasi-concave at every point of
   $N_\epsilon(a)$. The function $\epsilon^2 - \|b-a\|^2$ is clearly strictly
   quasi-concave on $N_\epsilon(a)$ and, by Theorem~\ref{henith},
   $N_\epsilon(a)$ is a spectrahedral shadow.

   The sets $N_\epsilon(a)$ cover the boundary of $\sS_e(p)$ and, by
   compactness, there is a finite subcover of $\partial\sS_e(p)$ by
   spectrahedral shadows. Using Theorem 2.2 in~\cite{heni2}, taking the
   convex hull of the finite cover shows that $\sS_e(p)$ is a
   spectrahedral shadow.
\end{proof}

The \emph{lineality space} $\lin(S)$ of a convex set $S \subset \R^n$ is the
largest linear subspace $L$ such that $a + L \subseteq S$ for some
(equivalently all) $a \in S$.  Assuming that $0 \in S$, it is
well-known (cf.~\cite[Thm.~2.5.8]{web}) that
\begin{equation}\label{linspc}
    S \ = \ (S \cap L^\perp) + L \ \subseteq \ L^\perp + L \ = \
    \R^n
\end{equation}
and since spectrahedral shadows are closed under taking Minkowski sums, it is
sufficient to prove our claims for pointed convex sets.

We can now easily translate the last result to hyperbolic polynomials, and thus prove Theorem \ref{mainhyp}. We only stated the smooth version as described in Remark \ref{rem}, and leave the more general version as in Theorem \ref{mainrz} to the reader.

\begin{proof}[Proof of Theorem~\ref{mainhyp}]
    As just explained, we can assume that $\Lambda_e(h)$ is pointed. Then
    there is a hyperplane $H$ such that $S = \Lambda_e(h) \cap H$ is compact
    and $\Lambda_e(h) = \cone(S)$. Since $H$ meets the interior of
    $\Lambda_e(h)$, we can assume that $e \in H$ and thus $S$ is a compact
    rigidly convex set defined by a real-zero polynomial meeting the
    requirements of Theorem~\ref{mainrz}. Finally, taking the conical hull
    retains the property of being a spectrahedral shadow; see for example,
    Proposition~2.1 in~\cite{nesi}.
\end{proof}

As mentioned in the introduction, any hyperbolic polynomial can be
approximated arbitrarily close by a strict one, see Nuij \cite{nuij}. For this
let $\partial_e h$ be the directional derivative of $h$ with respect to $e$.
By Rolle's Theorem, this is again a hyperbolic polynomial. For any linear form $\ell$ with $\ell(e)\neq 0$ and
$\epsilon > 0$ sufficiently small
\[
    \tilde{h}(x) \ := \ h(x)\ + \ \epsilon \ell(x)\partial_eh(x)
\]
is again hyperbolic, and the root multiplicity at each $a$ with $h(a)=0$ and
$\ell(a) \not=0$ is reduced by one.  In particular, if $\ell$ does not vanish
on $\Lambda_e(h)$, this reduces the multiplicity on the boundary of the
hyperbolicity cone. Iterating this process gives rise to a hyperbolic
polynomial that meets our smoothness assumption, and that is arbitrarily close
to $h$. In view of Theorem \ref{mainhyp}, each hyperbolicity cone can be
approximated arbitrarily close by a spectrahedral shadow.  We currently do not
know if our arguments can be extended to all hyperbolicity cones and we close
with the following conjecture.

\textbf{Projected Lax Conjecture. } Every hyperbolicity cone is a spectrahedral shadow.

\end{document}